\documentclass[12pt,A4]{amsart}
\usepackage{amsfonts,amsthm,amsmath}
\usepackage{rotating}
\usepackage{tikz}
\usepackage{verbatim}
\usepackage{amscd}
\usepackage[latin2]{inputenc}
\usepackage{t1enc}
\usepackage[mathscr]{eucal}
\usepackage{indentfirst}
\usepackage{graphicx}
\usepackage{graphics}
\usepackage{pict2e}
\usepackage{mathrsfs}
\usepackage{enumerate}
\usepackage[pagebackref]{hyperref}
\hypersetup{backref, pagebackref, colorlinks=true}
\usepackage{cite}
\usepackage{color}
\usepackage{epic}
\usepackage{hyperref} 
\usepackage{tkz-graph}

\numberwithin{equation}{section}
\def\blue{\textcolor{blue}}
\def\red{\textcolor{red}}

\theoremstyle{plain}
\newtheorem{theorem}{Theorem}[section]
\newtheorem{lemma}[theorem]{Lemma}

\newtheorem{proposition}[theorem]{Proposition}

\theoremstyle{definition}

\newtheorem{conjecture}[theorem]{Conjecture}
\newtheorem{remark}[theorem]{Remark}
\newtheorem{?}[theorem]{Problem}

\newcommand{\N}{\mathbb{N}}
\newcommand{\Z}{\mathbb{Z}}

\newcommand{\Q}{\mathbb{Q}}

\def\max{\mathrm{max}}

\def\AW{\mathcal{AW}}

\def\asc{\mathrm{asc}}

\def\max{\operatorname{max}}

\def\I{\operatorname{{\bf I}}}

\def\PT{\mathrm{PT}}
\def\PA{\mathcal{PA}}
\def\A{\mathcal{A}}

\def\boxit#1{\leavevmode\hbox{\vrule\vtop{\vbox{\kern.33333pt\hrule\kern1pt\hbox{\kern1pt\vbox{#1}\kern1pt}}\kern1pt\hrule}\vrule}}

\usepackage{collectbox}


\begin{document}

\title[Restricted inversion sequences]{Restricted inversion sequences and enhanced $3$-noncrossing partitions}

\author[Z. Lin]{Zhicong Lin}
\address[Zhicong Lin]{School of Science, Jimei University, Xiamen 361021, P.R. China
\& Fakult\"at f\"ur Mathematik, Universit\"at Wien, 1090 Wien, Austria} 
\email{zhicong.lin@univie.ac.at}

\date{\today}

\begin{abstract} We prove  a conjecture due independently to Yan and Martinez--Savage that asserts inversion sequences with no weakly decreasing subsequence of length $3$ and enhanced $3$-noncrossing partitions have the same cardinality. Our approach applies both the generating tree technique and the so-called obstinate  kernel method developed by Bousquet-M\'elou. One application of this  equinumerosity is a discovery of an intriguing identity involving numbers of classical and enhanced $3$-noncrossing partitions.
\end{abstract}


\maketitle


\section{Introduction}
Set partitions avoiding $k$-crossings and $k$-nestings have been extensively studied from the points of view of both  Combinatorics and Mathematical Biology; see~\cite{chen,chen2,kr0} and the references therein. The bijection between partitions and vacillating (resp.~hesitating) tableaux  due to Chen, Deng, Du, Stanley and Yan~\cite{chen} is now a fundamental tool for analyzing classical (resp.~enhanced) $k$-crossings and $k$-nestings. In particular, these two bijections were applied by Bousquet-M\'elou and Xin~\cite{bx} to enumerate set partitions avoiding classical or enhanced 3-crossings. After their work, the sequence $\{C_3(n)\}_{n\geq1}$ (resp.~$\{E_3(n)\}_{n\geq1}$) where $C_3(n)$ (resp.~$E_3(n)$) is the number of partitions of $[n]:=\{1,2,\ldots,n\}$ avoiding classical (resp.~enhanced) $3$-crossings  has been registered  as \href{https://oeis.org/A108304}{A108304} (resp.~\href{https://oeis.org/A108307}{A108307}) in OEIS:
\begin{align*}
\{C_3(n)\}_{n\geq1}&=\{1,2,5,15,52,202,859,3930,\ldots\},\\
\{E_3(n)\}_{n\geq1}&=\{1,2,5,15,51,191,772,3320,\ldots\}.
\end{align*}
The main purpose of this paper is to show that the sequence $\{E_3(n)\}_{n\geq1}$ also enumerates  inversion sequences with no weakly decreasing subsequence of length $3$. As we will see, this implies the following intriguing identity between $\{C_3(n)\}_{n\geq1}$ and $\{E_3(n)\}_{n\geq1}$: 
\begin{equation}\label{eq:3cros}
C_3(n+1)=\sum_{i=0}^n{n\choose i}E_3(i),
\end{equation}
where we use the convention $E_3(0)=1$.

It is convenient to recall some necessary definitions. For each $n\geq1$, let $\I_n$ be the set of {\em inversion sequences} of length $n$  defined as 
$$\I_n:=\{(e_1,e_2,\ldots,e_n): 0\leq e_i<i\}.$$
An inversion sequence $e\in\I_n$ is said to  be {\em$(\geq,\geq,-)$-avoiding} if there does not exist $i<j<k$ such that $e_i\geq e_j\geq e_k$.
The set of all $(\geq,\geq,-)$-avoiding inversion sequences in $\I_n$ is denoted by $\I_n(\geq,\geq,-)$. For example, we have 
$$\I_3(\geq,\geq,-)=\{(0,0,1),(0,0,2),(0,1,0),(0,1,1),(0,1,2)\}.$$
Recently, Martinez and Savage~\cite{ms} studied this class of restricted inversion sequences and suspected the following connection with enhanced $3$-noncrossing partitions. 

\begin{conjecture}[Yan~\cite{yan} \& Martinez--Savage~\cite{ms}]\label{conj:yan}
The cardinality of $\I_n(\geq,\geq,-)$ is $E_3(n)$. 
\end{conjecture}

In fact, this conjecture has already been proposed by Yan~\cite{yan} several years before in proving a conjecture of Duncan and Steingr\'imsson~\cite{ds}. 
We notice that in~\cite{yan} there is an interesting bijection  between $(\geq,\geq,-)$-avoiding inversion sequences and $210$-avoiding primitive ascent sequences as we review below. 

Recall that a sequence of integers $x=x_1x_2\cdots x_n$ is called an {\em ascent sequence} if it satisfies $x_1=0$ and for all $2\leq i\leq n$, $0\leq x_i\leq \asc(x_1x_2\cdots x_{i-1})+1$, where 
$$\asc(x_1x_2\cdots x_{i-1})=|\{j\in[i-2]:x_j<x_{j+1}\}|
$$ is the ascent number of $x_1x_2\cdots x_{i-1}$. 
Such an ascent sequence $x$ is said to be 
\begin{itemize}
\item {\em $210$-avoiding}, if $x$ does not have decreasing subsequence of length $3$;
\item {\em primitive}, if $x_i\neq x_{i+1}$ for all $i\in[n-1]$.
\end{itemize}
Denote by $\A_n(210)$ and $\PA_n(210)$ the set of all $210$-avoiding ordinary and primitive ascent sequences of length $n$, respectively. For example, we have
$$
\A_3(210)=\{000,001,010,011,012\}\text{ and }\PA_4(210)=\{0101,0102,0120,0121,0123\}.
$$
Via an intermediate structure of growth diagrams for $01$-fillings of Ferrers shapes, Yan~\cite{yan} proved combinatorially the following equinumerosity, which was first conjectured in~\cite[Conjecture~3.3]{ds}. 

\begin{theorem}[Main result of Yan~\cite{yan}] \label{thm:yan}
The cardinality of $\A_n(210)$ is $C_3(n)$. 
\end{theorem}

In the course of her combinatorial proof to Theorem~\ref{thm:yan}, she also showed that the mapping $\phi:\PA_{n+1}(210)\rightarrow\I_n(\geq,\geq,-)$ defined for each $x\in\PA_{n+1}(210)$ by
$$
\phi(x)=(e_1,e_2,\ldots,e_n),\,\,\text{where $e_i=i-1+x_{i+1}-\asc(x_1x_2\cdots x_{i+1})$},
$$   is a bijection. Therefore, Conjecture~\ref{conj:yan} is equivalent to $|\PA_{n+1}(210)|=E_3(n)$, as was originally suggested in~\cite[Remark~3.6]{yan}. 

 The rest of this paper is laid out as follows. In section~\ref{sec:tree}, we develop the generating tree for $\cup_{n\geq1}\I_n(\geq,\geq,-)$ and obtain a resulting functional equation. In Section~\ref{sec:proof}, we solve this functional equation via the obstinate kernel method~\cite{bo}  and then apply the Lagrange inversion formula and Zeilberger's algorithm to finish the proof of Conjecture~\ref{conj:yan}. In Section~\ref{sec:yan},  we show that how Conjecture~\ref{conj:yan} together with the results in~\cite{bx} would provide an alternative approach to Theorem~\ref{thm:yan}. An extension of~\eqref{eq:3cros} to $k$-noncrossing partitions is also conjectured. 
In Section~\ref{sec:aw}, we apply similar technique as in section~\ref{sec:tree} to enumerate      another interesting class of restricted inversion sequences introduced by Adams-Watters~\cite{aw}. It is surprising that the resulting functional equation is difficult enough that  we do not know how to solve it. 
Finally, we conclude the paper with some further remarks. 

\section{The generating tree for $(\geq,\geq,-)$-avoiding inversion sequences}
\label{sec:tree}

A {\em left-to-right maximum} of an inversion sequence $(e_1,e_2,\ldots,e_n)$ is an entry $e_i$ satisfying $e_i>e_j$ for any $j<i$. Similar to $321$-avoiding permutations,  $(\geq,\geq,-)$-avoiding inversion sequences have the following important characterization  proved by Martinez--Savage~\cite{ms}.

\begin{proposition}[See~\cite{ms}, Observation~7]\label{decr} An inversion sequence is  $(\geq,\geq,-)$-avoiding if and only if both the subsequence formed by its left-to-right maximum and the one formed by the remaining entries are strictly increasing. 
\end{proposition}

For each $e\in\I_n(\geq,\geq,-)$, introduce the {\em parameters} $(p,q)$ of $e$, where 
$$p=\alpha(e)-\beta(e) \quad{ and } \quad q=n-\alpha(e)
$$ with $\alpha(e)=\max\{e_1,e_2,\ldots,e_n\}$ and $\beta(e)$ is the greatest integer in the set
 $$\{e_i: e_i\text{ is not a left-to-right maximum}\}\cup\{-1\}.$$ 
 For example, the parameters of $(0,1,2)$ is $(3,1)$, while the parameters of $(0,1,1)$ is $(0,2)$.
We have the following rewriting rule for $(\geq,\geq,-)$-avoiding inversion sequences.

\begin{lemma}\label{lem:cross}
Let $e\in\I_n(\geq,\geq,-)$ be an inversion sequence with parameters $(p,q)$. Exactly $p+q$ inversion sequences in $\I_{n+1}(\geq,\geq,-)$ when removing their last entries will become $e$, and their parameters are respectively:
  \begin{align*}
 &(p-1,q+1), (p-2,q+1),\ldots, (0,q+1)\\
&(p+1,q), (p+2,q-1),\ldots, (p+q,1).
 \end{align*}
The order in which the parameters are listed corresponds to the inversion sequences with last entries from $\beta(e)+1$ to $n$. 
\end{lemma}
\begin{proof}
In view of Proposition~\ref{decr},  the vector $f:=(e_1,e_2,\ldots,e_n,b)$ is an inversion sequence in $\I_{n+1}(\geq,\geq,-)$ if and only if $\beta(e)<b\leq n$. We distinguish  two cases: 
\begin{itemize}
\item If $\beta(e)<b\leq\alpha(e)$, then $\alpha(f)=\alpha(e)$ and $\beta(f)=b$. These contribute the parameters $(p-1,q+1), (p-2,q+1),\ldots, (0,q+1)$. 
\item If $\alpha(e)<b\leq n$, then $\alpha(f)=b$ and $\beta(f)=\beta(e)$. This case contributes the parameters $(p+1,q), (p+2,q-1),\ldots, (p+q,1)$.
\end{itemize}
These two cases together give the rewriting rule for $(\geq,\geq,-)$-avoiding inversion sequences. Note that $p=\alpha(e)-\beta(e)$ may be $0$, i.e. $\alpha(e)=\beta(e)$, and in this situation the first case is empty. 
\end{proof}

Using the above lemma, we construct a {\em generating tree} (actually an infinite rooted tree) for $(\geq,\geq,-)$-avoiding inversion sequences by representing each element as its parameters as follows: the root is $(1,1)$ and the children of a vertex labelled $(p,q)$ are those generated according to the  rewriting rule in Lemma~\ref{lem:cross}. See Fig.~\ref{tree} for the first few levels of this generating tree. Note that the number of vertices at the $n$-th level of this tree is the cardinality of $\I_n(\geq,\geq,-)$.

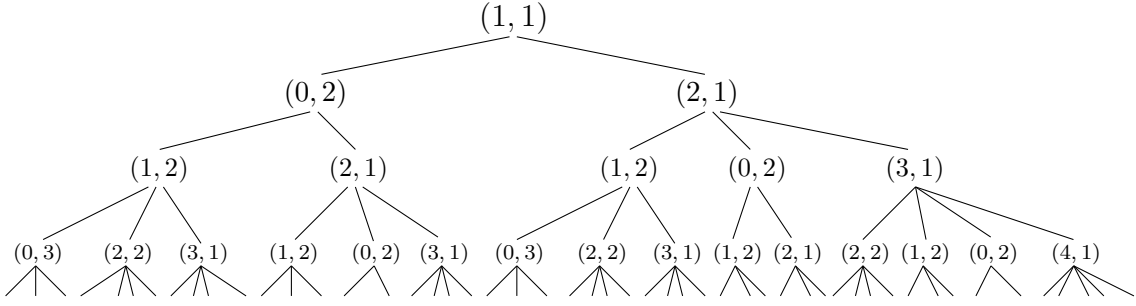
\begin{figure}
\setlength{\unitlength}{1mm}
\begin{picture}(120,38)\setlength{\unitlength}{1mm}
\thinlines
\put(50,35){$(1,1)$}
\put(54,33.5){\line(-5,-1){25}}\put(55,33.5){\line(5,-1){25}}
\put(24,25){\small{$(0,2)$}}\put(76,25){\small{$(2,1)$}}
\put(27.5,23.5){\line(-4,-1){20}}\put(28.5,23.5){\line(1,-1){5}}
\put(3.5,15){\footnotesize{$(1,2)$}}\put(30,15){\footnotesize{$(2,1)$}}
\put(6,13.5){\line(-2,-1){14}}\put(7,13.5){\line(-1,-2){3.5}}\put(8,13.5){\line(1,-1.5){5}}
\put(-12,4){\tiny{$(0,3)$}}\put(0,4){\tiny{$(2,2)$}}\put(10,4){\tiny{$(3,1)$}}

\put(-9,3){\line(-1,-1){4}}\put(-9,3){\line(0,-1){4}}\put(-9,3){\line(1,-1){4}}
\put(3,3){\line(-1,-4){1}}\put(3,3){\line(-3,-2){6}}\put(3,3){\line(1,-4){1}}\put(3,3){\line(1,-1){4}}
\put(13,3){\line(-1,-1){4}}\put(13,3){\line(-1,-4){1}}\put(13,3){\line(1,-4){1}}\put(13,3){\line(3,-2){6}}

\put(32.5,13.5){\line(-1,-1){7}}\put(33.5,13.5){\line(1,-3){2.4}}\put(34.5,13.5){\line(3,-2){10}}
\put(22,4){\tiny{$(1,2)$}}\put(33,4){\tiny{$(0,2)$}}\put(42,4){\tiny{$(3,1)$}}

\put(25,3){\line(-1,-1){4}}\put(25,3){\line(0,-1){4}}\put(25,3){\line(1,-1){4}}
\put(36,3){\line(-1,-1){4}}\put(36,3){\line(1,-2){2}}
\put(45,3){\line(-1,-4){1}}\put(45,3){\line(-1,-1){4}}\put(45,3){\line(1,-4){1}}\put(45,3){\line(1,-1){4}}

\put(80,23.5){\line(-2,-1){10}}\put(81,23.5){\line(1,-1){5}}\put(82,23.5){\line(5,-1){25}}
\put(66,15){\footnotesize{$(1,2)$}}\put(83,15){\footnotesize{$(0,2)$}}\put(104,15){\footnotesize{$(3,1)$}}

\put(69,13.5){\line(-2,-1){14}}\put(70,13.5){\line(-1,-2){3.5}}\put(71,13.5){\line(1,-1.5){5}}
\put(52,4){\tiny{$(0,3)$}}\put(63,4){\tiny{$(2,2)$}}\put(73,4){\tiny{$(3,1)$}}

\put(55,3){\line(-1,-1){4}}\put(55,3){\line(0,-1){4}}\put(55,3){\line(1,-1){4}}
\put(66,3){\line(-1,-4){1}}\put(66,3){\line(-1,-1){4}}\put(66,3){\line(1,-4){1}}\put(66,3){\line(1,-1){4}}
\put(76,3){\line(-1,-4){1}}\put(76,3){\line(-1,-1){4}}\put(76,3){\line(1,-4){1}}\put(76,3){\line(1,-1){4}}

\put(86,13.5){\line(-1,-3){2.3}}\put(87,13.5){\line(2,-3){4.5}}
\put(81,4){\tiny{$(1,2)$}}\put(89,4){\tiny{$(2,1)$}}

\put(84,3){\line(-1,-2){2}}\put(84,3){\line(1,-2){2}}\put(84,3){\line(1,-1){4}}
\put(92,3){\line(-1,-2){2}}\put(92,3){\line(1,-2){2}}\put(92,3){\line(1,-1){4}}

\put(108,13.5){\line(-1,-1){7}}\put(108,13.5){\line(0.2,-1){1.4}}\put(108,13.5){\line(1.5,-1){10}}\put(108,13.5){\line(3,-1){21}}
\put(98,4){\tiny{$(2,2)$}}\put(106,4){\tiny{$(1,2)$}}\put(115,4){\tiny{$(0,2)$}}\put(126,4){\tiny{$(4,1)$}}

\put(101,3){\line(-1,-4){1}}\put(101,3){\line(-1,-1){4}}\put(101,3){\line(1,-4){1}}\put(101,3){\line(1,-1){4}}
\put(109,3){\line(-1,-2){2}}\put(109,3){\line(1,-2){2}}\put(109,3){\line(1,-1){4}}
\put(118,3){\line(1,-1){4}}\put(118,3){\line(-1,-2){2}}
\put(129,3){\line(-1,-1){4}}\put(129,3){\line(1,-2){2}}\put(129,3){\line(-1,-2){2}}\put(129,3){\line(1,-1){4}}\put(129,3){\line(2,-1){8}}
\end{picture}
\caption{First few levels of the generating tree for $\cup_{n\geq1}\I_n(\geq,\geq,-)$.}
\label{tree}
\end{figure}

Define the formal power series $E(t;u,v)=E(u,v):=\sum_{p\geq0,q\geq1}E_{p,q}(t)u^pv^q$, where $E_{p,q}(t)$ is the size generating function for the $(\geq,\geq,-)$-inversion sequences with parameters $(p,q)$. We can turn this generating tree into a functional equation as follows. 
\begin{proposition}We have the following functional equation for $E(u,v)$:
\begin{equation}\label{eq:cross}
\biggl(1+\frac{tv}{1-u}+\frac{tv}{1-v/u}\biggr)E(u,v)=tuv+\frac{tv}{1-u}E(1,v)+\frac{tv}{1-v/u}E(u,u).
\end{equation}
\end{proposition}

\begin{proof}
In the generating tree for $\cup_{n\geq1}\I_n(\geq,\geq,-)$, each vertex other than the root $(1,1)$ can be generated by a unique parent. Thus, we have 
\begin{align*}
E(u,v)&=tuv+t\sum_{p\geq0,q\geq1}E_{p,q}(t)\biggl(v^{q+1}\sum_{i=0}^{p-1}u^i+\sum_{i=0}^{q-1}u^{p+1+i}v^{q-i}\biggr)\\
&=tuv+t\sum_{p\geq0,q\geq1}E_{p,q}(t)\biggl(\frac{1-u^p}{1-u}v^{q+1}+u^{p+1}v^q\frac{1-(u/v)^q}{1-u/v}\biggr)\\
&=tuv+\frac{tv}{1-u}(E(1,v)-E(u,v))+\frac{tuv}{v-u}(E(u,v)-E(u,u)),
\end{align*}
which is equivalent to~\eqref{eq:cross}.
\end{proof}

\begin{remark}
It should be noted that the kernel $1+\frac{tv}{1-u}+\frac{tv}{1-v/u}$ of~\eqref{eq:cross} is exactly the same as that of the functional equation for {\em Baxter inversion sequences} in~\cite[Proposition~4.4.]{kl}. 
\end{remark}

\section{Proof of Conjecture~\ref{conj:yan}}
\label{sec:proof}

In this section, we will prove Conjecture~\ref{conj:yan} by solving~\eqref{eq:cross}. It is convenient to set $v=uw$ in~\eqref{eq:cross}. The equation then becomes
$$
\biggl(1+\frac{tuw}{1-u}+\frac{tuw}{1-w}\biggr)E(u,wu)=tu^2w+\frac{tuw}{1-u}E(1,uw)+\frac{tuw}{1-w}E(u,u).
$$
Further setting $u=1+x$ and $w=1+y$ above we get
\begin{multline}\label{eq2:cross}
\frac{xy-t(1+x)(1+y)(x+y)}{t(1+x)(1+y)}E(1+x,(1+x)(1+y))\\
=xy(1+x)-yE(1,(1+x)(1+y))-\widetilde{E}(x),
\end{multline}
where $\widetilde{E}(x)=xE(1+x,1+x)$. We are going to apply the {\em obstinate kernel method} developed by Bousquet-M\'elou~\cite{bo} to this equation. The numerator 
$$
K(x,y)=xy-t(1+x)(1+y)(x+y) 
$$
of the coefficient of $E(1+x,(1+x)(1+y))$ in~\eqref{eq2:cross} is called the {\em kernel} of~\eqref{eq2:cross}.

Observe that $K(x,y)$ is also the kernel of the functional equation in~\cite[Corollary~3]{bo} for {\em Baxter permutations}.
It was shown in~\cite[Figure~3]{bo} that the three pairs $(x,Y), (\bar{x}Y,Y)$ and $(\bar{x}Y,\bar{x})$ are roots of the kernel $K(x,y)$ and can be legally substituted for $(x,y)$ in~\eqref{eq2:cross}, where 
$$
\bar{x}:=\frac{1}{x}\quad\text{and}\quad Y=\frac{1-t(1+x)(1+\bar{x})-\sqrt{1-2t(1+x)(1+\bar{x})-t^2(1-x^2)(1-\bar{x}^2)}}{2t(1+\bar{x})}.
$$
Note that the  kernel $K(x,y)$ is symmetric in $x$ and $y$ and  so the dual pairs $(Y,x), (Y,\bar{x}Y)$ and $(\bar{x},\bar{x}Y)$ are also roots of  $K(x,y)$ which can be legally substituted for $(x,y)$ in~\eqref{eq2:cross}. Substituting the pairs $(x,Y)$ and $(Y,x)$ for $(x,y)$ in~\eqref{eq2:cross} yields 
$$
\begin{cases}
\,\,xY(1+x)-YE(1,(1+x)(1+Y))-\widetilde{E}(x)=0,
\\
\,\,Yx(1+Y)-xE(1,(1+x)(1+Y))-\widetilde{E}(Y)=0.
\end{cases}
$$
Eliminating $E(1,(1+x)(1+Y))$ we get 
\begin{equation}\label{eq:1}
Y\widetilde{E}(Y)-x\widetilde{E}(x)=xY(Y(1+Y)-x(1+x)).
\end{equation}
Similarly, substitute $(\bar{x}Y,Y),(Y,\bar{x}Y)$ and $(\bar{x}Y,\bar{x}),(\bar{x},\bar{x}Y)$ into~\eqref{eq2:cross} and after some computation we get two equations, which together with~\eqref{eq:1} give the system of equations:
$$
\begin{cases}
\,\,Y\widetilde{E}(Y)-x\widetilde{E}(x)=xY(Y(1+Y)-x(1+x)),
\\
\,\,Y\widetilde{E}(Y)-\bar{x}Y\widetilde{E}(\bar{x}Y)=\bar{x}Y^2(Y(1+Y)-\bar{x}Y(1+\bar{x}Y)),
\\
\,\,\bar{x}\widetilde{E}(\bar{x})-\bar{x}Y\widetilde{E}(\bar{x}Y)=\bar{x}^2Y(\bar{x}(1+\bar{x})-\bar{x}Y(1+\bar{x}Y)).
\end{cases}
$$
By eliminating $\widetilde{E}(Y)$ and $\widetilde{E}(\bar{x}Y)$, we get a relation between $\widetilde{E}(x)$ and $\widetilde{E}(\bar{x})$:
\begin{equation}\label{eq:main}
\bar{x}\widetilde{E}(x)-\bar{x}^3\widetilde{E}(\bar{x})=R(x,Y), 
\end{equation}
where 
\begin{equation}\label{def:R}
R(x,Y)=Y(x+1-\bar{x}^{5}-\bar{x}^{6})+Y^2(\bar{x}^{5}-\bar{x})+Y^3(\bar{x}^{3}+\bar{x}^{6}-\bar{x}-\bar{x}^{4})+Y^4(\bar{x}^{3}-\bar{x}^{5})
\end{equation}
is a formal power series in $t$.
Since in the left-hand side of~\eqref{eq:main}: 
\begin{itemize}
\item $\bar{x}\widetilde{E}(x)=E(1+x,1+x)$  is a power series in $t$ with polynomial coefficient in $x$ 
\item and $\bar{x}^3\widetilde{E}(\bar{x})$ is a power series in $t$ with polynomial coefficient in $\bar{x}$ whose lowest power of $\bar{x}$ is $4$,
\end{itemize}
we have the following result.
\begin{theorem}
Let $Y=Y(t;x)$ be the unique formal power series in $t$ such that
\begin{equation}\label{eq:Y}
Y=t(1+\bar{x})(1+Y)(x+Y).
\end{equation}
The series solution $E(u,v)$ of~\eqref{eq:cross} satisfies 
\begin{equation}\label{eq:ER}
E(1+x,1+x)=\mathop{\PT}\limits_{x} R(x,Y),
\end{equation}
where $R(x,Y)$ is defined in~\eqref{def:R} and the operator $\mathop{\PT}\limits_{x}$ extracts non-negative powers of $x$ in series of $\Q[x,\bar{x}][[t]]$.
\end{theorem}

Now we can apply the {\em Lagrange inversion formula} and {\em Zeilberger's algorithm} to finish the proof of Conjecture~\ref{conj:yan}. 
\begin{proof}[{\bf Proof of Conjecture~\ref{conj:yan}}]
Let $E(n)=|\I_n(\geq,\geq,-)|$. 
It follows from~\eqref{eq:ER} that 
\begin{multline}\label{eq:En}
E(n)=[x^{-1}t^n]Y+[x^0t^n]Y-[x^5t^n]Y-[x^6t^n]Y+[x^5t^n]Y^2-[x^1t^n]Y^2\\
+[x^3t^n]Y^3+[x^6t^n]Y^3-[x^1t^n]Y^3-[x^4t^n]Y^3+[x^3t^n]Y^4-[x^5t^n]Y^4.
\end{multline}
Applying the Lagrange inversion formula~\cite[Theorem~5.4.2]{st2}
to~\eqref{eq:Y} gives:
\begin{align*}
[x^mt^n]Y^k&=\frac{k}{n}[x^mt^{n-k}]((x+t)(1+t)(1+\bar{x}))^n\\
&=\frac{k}{n}\sum_{i=0}^{n-k}{n\choose i}{n\choose k+i}{n\choose m+i}
\end{align*}
for all $k,m\in\Z$ and $n\in\N$. Substituting this into~\eqref{eq:En} we can express $E(n)$ as $E(n)=\sum_{i=0}^{n-1}E(n,i)$, where 
\begin{multline*}
E(n,i)=\frac{1}{n}{n\choose i}\left\{{n\choose i+1}\left[{n+1\choose i}-{n+1\choose i+6}\right]+2{n\choose i+2}\left[{n\choose i+5}-{n\choose i+1}\right]\right.\\
\left.+3{n\choose i+3}\left[{n\choose i+3}+{n\choose i+6}-{n\choose i+1}-{n\choose i+4}\right]+4{n\choose i+4}\left[{n\choose i+3}-{n\choose i+5}\right]\right\}.
\end{multline*}
Applying Zeilberger's  algorithm~\cite{pwz} (or creative telescoping) with $E(n,i)$ above as input, the Maple package {\tt ZeilbergerRecurrence(E(n,i),n,i,E,0..n-1)} gives the P-recursion: for $n\geq1$,
\begin{equation}\label{recu:En}
a_nE(n)+b_nE(n+1)+c_nE(n+2)-d_nE(n+3)=0,
\end{equation}
where 
\begin{align*}
a_n&=8(3n+13)(n+3)(n+2)(n+1),\\
b_n&=3(n+3)(n+2)(15n^2+153n+376),\\
c_n&=6(n+7)(3n^3+38n^2+156^n+212),\\
d_n&=(3n+10)(n+9)(n+8)(n+7).
\end{align*}
The initial conditions are $E(1)=1, E(2)=2$ and $E(3)=5$.

On the other hand, Bousquet-M\'elou and Xin~\cite[Proposition~2]{bx} showed that the number $E_3(n)$ satisfies the P-recursion: $E_3(0)=E_3(1)=1$, and for $n\geq0$,
\begin{equation}\label{recu:E3}
8(n+3)(n+1)E_3(n)+(7n^2+53n+88)E_3(n+1)-(n+8)(n+7)E_3(n+2)=0.
\end{equation}
It is then routine to check that the sequence defined by the above three term recursion satisfies also the four term recursion in~\eqref{recu:En} obtained via Zeilberger's algorithm. More precisely, applying to~\eqref{recu:E3} the operator 
$$
(3n+13)(n+2)+(3n+10)(n+7)N,
$$
where $N$ is the shift operator replacing $n$ by $n+1$, yields a four term recursion for $E_3(n)$ which is exactly the same as that for $E(n)$ in~\eqref{recu:En}.  This completes the proof of Conjecture~\ref{conj:yan}, since both sequences share the same initial values.
\end{proof}

Since our proof of Conjecture~\ref{conj:yan} uses formal power series heavily, it is natural to ask for a bijective proof.

\section{A new approach to Yan's result and a conjecture}
\label{sec:yan}

Let $x=x_1x_2\cdots x_{n+1}$ be a $210$-avoiding ascent sequence of length $n+1$. It is apparent that the ascent sequence $x$ can be written uniquely as $\tilde{x}_1^{c_1} \tilde{x}_2^{c_2}\cdots \tilde{x}_{i+1}^{c_{i+1}}$, where $\tilde{x}:=\tilde{x}_1\tilde{x}_2\cdots\tilde{x}_{i+1}$ is a $210$-avoiding  primitive ascent sequence of length $i+1$ and $c_1+c_2+\cdots+c_{i+1}=n+1$ is a $(i+1)$-composition of $n+1$. For instance, the ascent sequence $0110212224\in\A_{10}(210)$ can be written as $0^11^20^12^11^12^34^1$, so that $\tilde{x}=0102124\in\PA_7(210)$ and the corresponding $7$-composition is $1+2+1+1+1+3+1=10$. Since the number of  $(i+1)$-composition of $n+1$ is ${n\choose i}$, the above decomposition gives the identity:
$$
|\A_{n+1}(210)|=\sum_{i=0}^n{n\choose i}|\PA_{i+1}(210)|=\sum_{i=0}^n{n\choose i}E_3(i),
$$
where the second equality follows from $|\PA_{i+1}(210)|=E_3(i)$ (by Conjecture~\ref{conj:yan}). Therefore, Theorem~\ref{thm:yan} is equivalent to identity~\eqref{eq:3cros}.  In the following, we will show how to deduce~\eqref{eq:3cros} from the results in~\cite{bx}, which provides a new approach to Theorem~\ref{thm:yan}.

Let $\mathcal{C}(t)=\sum_{n\geq1}\sum_{i=0}^{n-1}{n-1\choose i}E_3(i)t^n$ and $\mathcal{E}(t)=\sum_{n\geq0}E_3(n)t^n$. It then follows that  
\begin{equation}\label{eq:CE}
\mathcal{C}(t)=\sum_{i\geq0}E_3(i)t^{i+1}\sum_{m\geq0}{m+i\choose i}t^{m}=\sum_{i\geq0}E_3(i)\biggl(\frac{t}{1-t}\biggr)^{i+1}=z\mathcal{E}(z),
\end{equation}
where $z=\frac{t}{1-t}$. As was shown in~\cite[Proposition~2]{bx}, the generating function $\mathcal{E}(t)$ satisfies: 
$$
t^2(1+t)(1-8t)\frac{d^2}{dt^2}\mathcal{E}(t)+2t(6-23t-20t^2)\frac{d}{dt}\mathcal{E}(t)+6(5-7t-4t^2)\mathcal{E}(t)=30.
$$
Thus, if we denote $\mathcal{F}(t)=t\mathcal{E}(t)$, then 
\begin{equation}\label{eq:FF}
t^2(1+t)(1-8t)\frac{d^2}{dt^2}\mathcal{F}(t)+2t(5-16t-12t^2)\frac{d}{dt}\mathcal{F}(t)+(20-10t)\mathcal{F}(t)=30t.
\end{equation}
By~\eqref{eq:CE}, we have  $\mathcal{F}(t)=\mathcal{C}(x)$ with $x=\frac{t}{1+t}$. Substituting $\mathcal{F}(t)=\mathcal{C}(x)$ into~\eqref{eq:FF} and using the chain rule, we get 
$$
x^2(1-9x)(1-x)\frac{d^2}{dx^2}\mathcal{C}(x)+2x(5-27x+18x^2)\frac{d}{dx}\mathcal{C}(x)+10(2-3x)\mathcal{C}(x)=30x
$$
after some manipulation. 
Comparing with~\cite[Proposition~1]{bx} we conclude that 
$\mathcal{C}(t)=\sum_{n\geq1}C_3(n)t^n$, which is equivalent to~\eqref{eq:3cros}, as desired.

\subsection{Extension of~\eqref{eq:3cros} to $k$-noncrossing partitions: a conjecture}

\begin{figure}
\begin{center}
\begin{tikzpicture}
\SetVertexNormal
\SetGraphUnit{1.3}
\tikzset{VertexStyle/.append style={inner sep=0pt,minimum size=5mm}}
\Vertex{1}
\EA(1){2}\EA(2){3}\EA(3){4}\EA(4){5}\EA(5){6}\EA(6){7}\EA(7){8}
\tikzset{EdgeStyle/.append style = {bend left = 60}}
\red{\Edge(1)(3)\Edge(2)(5)}
\Edge(5)(6)
\Edge(3)(7)\Edge(6)(8)
\end{tikzpicture}
\end{center}
\caption{The arc diagram of $\{\{1,3,7\},\{2,5,6,8\},\{4\}\}$.\label{arc}}
\end{figure}
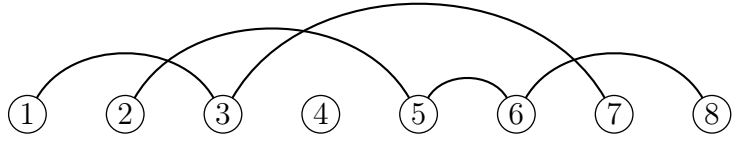

Any partition $P$ of $[n]$ can be identified with its {\em arc diagram} defined as follows:
\begin{itemize}
\item put the nodes $1,2,\ldots,n$ on a horizontal line in increasing order;
\item then draw an arc from $i$ to $j$, $i<j$, whenever $i$ and $j$ belong to a same block of $P$ and inside this block, there is not any $l$ satisfying $i<l<j$. 
\end{itemize}
See Fig.~\ref{arc} for the arc diagram of $\{\{1,3,7\},\{2,5,6,8\},\{4\}\}$. For any $k\geq2$, a {\em$k$-crossing} (resp.~an {\em enhanced $k$-crossing}) of  $P$ is a $k$-subset $(i_1,j_1), (i_2,j_2),\ldots, (i_k,j_k)$ of arcs in the arc diagram of $P$ such that 
$$
i_1<i_2<\cdots<i_k<j_1<j_2<\cdots<j_k\,\,\text{(resp.~$i_1<i_2<\cdots<\blue{i_k\leq j_1}<j_2<\cdots<j_k$)}.
$$
For instance, the partition in  Fig.~\ref{arc} has no $3$-crossing but contains one enhanced $3$-crossing, which is formed by the arcs $(1,3), (2,5), (3,7)$.

Let $C_k(n)$ (resp.~$E_k(n)$) be the number of partitions of $[n]$ avoiding classical (resp.~enhanced) $k$-crossings. It is known that $C_2(n)=C_n:=\frac{1}{n+1}{2n\choose n}$, the $n$th  {\em Catalan number}, and $E_2(n)=\sum\limits_{i=0}^{\lfloor n/2\rfloor}{n\choose 2i}C_i$ is the $n$th {\em Motzkin number}~\cite[Exercise~6.38]{st2}. The Catalan numbers are also related to Motzkin numbers by (cf.~\cite{deng})
\begin{equation} \label{ca:mz}
C_2(n+1)=\sum_{i=0}^n{n\choose i}E_2(i).
\end{equation}
In other words, the binomial transformation of Motzkin numbers are Catalan numbers. In view of  identities~\eqref{ca:mz} and~\eqref{eq:3cros}, the following conjecture is tempting.

\begin{conjecture}\label{conj:lin}
Fix $k\geq2$. The following identity holds:
\begin{equation}\label{eq:kimlin}
C_k(n+1)=\sum_{i=0}^n{n\choose i}E_k(i).
\end{equation}
\end{conjecture}

It would be interesting to see if the bijections of Chen et al.~\cite{chen} or Krattenthaler~\cite{kr0} could help to prove this conjecture. If this conjecture is true, then the $D$-finiteness (see~\cite[Theorem~6.4.10]{st2}) of 
$$
\mathcal{C}_k(t)=\sum_{n\geq1}C_k(n)t^n\quad\text{and}\quad \mathcal{E}_k(t)=\sum_{n\geq1}E_k(n)t^n
$$
are the same.

\section{Adams-Watters' restricted inversion sequences}%
\label{sec:aw}

An inversion sequence $e=(e_1,e_2,\ldots,e_n)\in\I_n$ is called a {\em $\AW$-inversion sequence}  (here $\AW$ stands for Adams-Watters) if for every $2<i\leq n$, we have $e_i\leq \max\{e_{i-2},e_{i-1}\}+1$. Let $\I_n(\AW)$ denote the set of $\AW$-inversion sequences of length $n$. 
For example, we have
$$
\I_3(\AW)=\{(0,0,0),(0,0,1),(0,1,0)(0,1,1)(0,1,2)\}.
$$
The $\AW$-inversion sequences were introduced by Adams-Watters~\cite{aw} (see also~\href{https://oeis.org/A108307}{A108307} in OEIS) who also conjectured that $|\I_n(\AW)|=E_3(n)$. Unfortunately, this is not true as
$$
\{|\I_n(\AW)|\}_{n\geq1}=\{1,2,5,15,191,773,3336,\ldots\}
$$
and this sequence now appears as \href{https://oeis.org/A275605}{A275605} in OEIS. We will show in the following how to get a functional equation for the generating function of a two-variable extension of this sequences.

In order to get a rewriting rule for $\AW$-inversion sequences, we introduce the {\em parameters} $(p,q)$ for each $e\in\I_n(\AW)$ by 
$$
p=e_n+1 \quad{ and } \quad q=\max\{e_{n-1},e_{n}\}+1-e_n.
$$
For example, the parameters of $(0,1,2,3,2)\in\I_5(\AW)$ is $(3,2)$. The following result can be checked routinely.

\begin{lemma}\label{lem:AW}
Let $e\in\I_n(\AW)$ be an inversion sequence with parameters $(p,q)$. Exactly $p+q$ inversion sequences in $\I_{n+1}(\AW)$ when removing their last entries will become $e$, and their parameters are respectively:
  \begin{align*}
 &(1,p), (2,p-1),\ldots, (p,1)\\
&(p+1,1), (p+2,1),\ldots, (p+q,1).
 \end{align*}
The order in which the parameters are listed corresponds to the inversion sequences with last entries from $0$ to $\max\{e_{n-1},e_{n}\}+1$. 
\end{lemma}

Define the formal power series $F(t;u,v)=F(u,v):=\sum_{p,q\geq1}F_{p,q}(t)u^pv^q$, where $F_{p,q}(t)$ is the size generating function for the $\AW$-inversion sequences with parameters $(p,q)$. We can translate Lemma~\ref{lem:AW} into the following   functional equation. 
\begin{proposition}We have the following functional equation for $F(u,v)$:
\begin{equation}\label{eq:AW}
F(u,v)=tuv+\frac{tuv}{v-u}(F(v,1)-F(u,1))+\frac{tuv}{1-u}(F(u,1)-F(u,u)).
\end{equation}
Equivalently, if we write $F(u,v)=\sum_{n\geq1}f_n(u,v)t^n$, then $f_1(u,v)=uv$ and for $n\geq2$,
\begin{equation}\label{rec:AW}
f_n(u,v)=\frac{uv}{v-u}(f_{n-1}(v,1)-f_{n-1}(u,1))+\frac{uv}{1-u}(f_{n-1}(u,1)-f_{n-1}(u,u)).
\end{equation}
\end{proposition}

Although we have not been able to solve~\eqref{eq:AW}, we note that recursion~\eqref{rec:AW} can be applied to compute $|\I_n(\AW)|=f_n(1,1)$ easily. 

\section{Final remarks}

Fix a positive integer $k$. The definition of $\AW$-inversion sequences can be generalized  to {\em $k$-$\AW$-inversion sequences} by requiring 
$$
e_i\leq \max\{e_{i-1},e_{i-2},\ldots,e_{i-k}\}+1 
$$
for an inversion sequence $e=(e_1,e_2,\ldots,e_n)$ and every $1\leq i\leq n$, where we take the convention $e_m=0$ whenever $m$ is nonpositive. It is apparent that $k$-$\AW$-inversion sequences of length $n$ is enumerated by
\begin{itemize}
\item the $n$th Catalan number $C_n$, when $k=1$;
\item the $n$th Bell number $B_n$, when $k=n-1$. Note that in this case, the $k$-$\AW$-inversion sequences are known as {\em restricted growth functions}, which are used to encode set partitions. 
\end{itemize}
The $2$-$\AW$-inversion sequences is just the $\AW$-inversion sequences we have investigated here. But even for enumeration of this special case, we have obtained no explicit formula. 

The longest decreasing and increasing subsequences and their variants in permutations have already been studied  from various aspects; see the interesting survey written by Stanley~\cite{st}. We expect similar studies on inversion sequences and ascent sequences to be fruitful. In particular,    
our results suggest that inversion sequences with no weakly $k$-decreasing subsequence and ascent sequences, primitive or ordinary,  avoiding strictly $k$-decreasing subsequence for $k>3$ may be worth further investigation.

\subsection*{Recent developments} Since a preliminary version of this paper was  posed on arXiv, there have been two interesting developments. Via $01$-filling of triangular shape, Yan~\cite{yan2} constructed a bijection between enhanced $3$-nonnesting partitions and $(\geq,\geq,-)$-avoiding inversion sequences, thereby providing a bijective proof of Conjecture~\ref{conj:yan}. Very recently, Kim and the author~\cite{kl2} obtained two different combinatorial proofs of 
Conjecture~\ref{conj:lin}, one of which even proves a refinement of~\eqref{eq:kimlin}, taking  the number of blocks  into account.

\section*{Acknowledgement} The author is grateful to Shaoshi Chen for his help on Zeilberger's algorithm. He also would like to thank the referees for their corrections and suggestions to improve the presentation. This work was done while the author was a Postdoc at CAMP, National Institute for Mathematical Sciences. The author's research was  supported by the National Science Foundation of China grant 11501244 and the Austrian Science Foundation FWF, START grant Y463 and SFB grant F50.


\begin{thebibliography}{99}

\bibitem{aw} F.T.  Adams-Watters, Personal communication, March 2017.

\bibitem{bo}M. Bousquet-M\'elou, Four classes of pattern-avoiding permutations under one roof: generating trees with two labels, Electron. J. Combin., {\bf9} (2003), \#R19. 

\bibitem{bx}M. Bousquet-M\'elou and G.C. Xin, On partitions avoiding $3$-crossings, S\'em. Lothar. Combin., {\bf54} (2006), Article B54e.

\bibitem{chen} W.Y.C. Chen, E.Y.P. Deng,  R.R.X. Du, R.P. Stanley and C.H. Yan, Crossings and nestings of matchings and partitions, Trans. Amer. Math. Soc., {\bf359} (2007), 1555--1575.

\bibitem{chen2} W.Y.C. Chen, J. Qin and C.M. Reidys, Crossings and nestings in tangled diagrams, Electron. J. Combin., {\bf15} (2008),  \#R86.

\bibitem{deng} E.Y.P. Deng and W.J. Yan, Some identities on the Catalan, Motzkin and Schr\"oder numbers, Discrete Appl. Math., {\bf156} (2008),  2781--2789.
 
\bibitem{ds} P. Duncan and E. Steingr\'imsson, Pattern avoidance in ascent sequences, Electron. J. Combin., {\bf18} (2011), \#P226.

\bibitem{kl} D. Kim and Z. Lin, Refined restricted inversion sequences (extended abstract at FPSAC 2017),  S\'em. Lothar. Combin., {\bf78B} (2017), Art. 52, 12pp. 

  
  \bibitem{kr0} C. Krattenthaler, Growth diagrams, and increasing and decreasing chains in fillings of Ferrers shapes, Adv. Appl. Math., {\bf 37} (2006), 404--431.
  
  \bibitem{kl2} Z. Lin and D. Kim, A combinatorial bijection on $k$-noncrossing partitions, in preparation. 

\bibitem{ms} M.A. Martinez and C.D. Savage, Patterns in Inversion Sequences II: Inversion Sequences Avoiding Triples of Relations,  \href{https://arxiv.org/abs/1609.08106}{arXiv:1609.08106}.

\bibitem{pwz}M. Petkovsek, H.S. Wilf and D. Zeilberger, {\em A=B}, A K Peters Ltd., Wellesley, MA, 1996.

\bibitem{st2} R.P. Stanley, {\em Enumerative Combinatorics}, vol. 2, Cambridge University Press, Cambridge, 1999.

\bibitem{st}R.P. Stanley, Increasing and decreasing subsequences and their variants, International Congress of Mathematicians. Vol. I, 545--579, Eur. Math. Soc., Z\"urich, 2007. 

\bibitem{yan} S.H.F. Yan, Ascent sequences and $3$-nonnesting set partitions,  European J. Combin., {\bf39} (2014), 80--94.

\bibitem{yan2} S.H.F. Yan, Bijections for inversion sequences, ascent sequences and $3$-nonnesting set partitions, \href{https://arxiv.org/abs/1707.02408v1}{arXiv:1707.02408v1}.

\end{thebibliography}
\end{document}